\documentclass[11pt]{article}
\usepackage{amssymb,amsmath,amsthm,amscd,amsfonts }
\usepackage{enumerate}

\newtheorem{thm}{Theorem}[section]
\newtheorem{cor}[thm]{Corollary}
\newtheorem{lem}[thm]{Lemma}

\newtheorem{prop}[thm]{Proposition}
\newtheorem{quest}[thm]{Question}

\theoremstyle{definition}
\newtheorem{defn}[thm]{Definition}

\numberwithin{equation}{section}

\newcommand{\diam}{\operatorname{diam}}
\newcommand{\im}{\operatorname{im}}
\renewcommand{\phi}{\varphi}
\newcommand{\cl}{\operatorname{cl}}
\renewcommand{\int}[1]{\operatorname{int}\bigl(#1\bigr)}
\renewcommand{\S}{\textbf{S}}

\begin{document}


\baselineskip=17pt


\date{}
\title{Homomorphims of fundamental groups of planar continua}
\author{Curtis Kent}
\maketitle



\begin{abstract}
We prove that every homomorphism from the fundamental group of a planar Peano continuum to the fundamental group of a planar  or one-dimensional Peano continuum is induced by a continuous map up to conjugation.  

This is then used to provide a family of uncountable many planar Peano continua with pairwise non-isomorphic fundamental groups all of which are not homotopy equivalent to a one-dimensional space. 
\end{abstract}



\setcounter{section}{0}

\section{Introduction }

Every continuous map between topological spaces induces a homomorphism of their fundamental groups.  It is a classical result that the converse also holds for CW-complexes and simplicial complexes.  However for more complicated spaces the converse can be false.  For example any inner automorphism of the fundamental group of a one-dimensional space which is not locally simply connected cannot be induced by a continuous map.   Katsuya Eda  was the first to prove the converse does hold upto conjugation when considering homomorphisms between fundamental groups of one-dimensional Peano continua \cite{eda}.

\textbf{Theorem A} \emph{ Let $X$ be one-dimensional Peano continua and $Y$ any one-dimensional metric space.  For every homomorphism $\phi: \pi_1(X,x_0)\to \pi_1(Y,y_0)$, there exists a path $\alpha: (I,0,1)\to (Y,y_0,y)$ and a continuous function $f:X\to Y$ such that $\hat \alpha\circ\phi = f_*$ where $\hat \alpha$ is the change of base point isomorphism induced by the path $\alpha$. In addition, if $\phi$ has uncountable image then $\alpha$ is unique up to homotopy rel endpoints. }


Understanding the extent to which homomorphisms of fundamental groups are induced by continuous maps on the underlying topological spaces provides an additional tool to relate the homotopy type of locally complicated spaces with their fundamental groups. (See  \cite{cc3}, \cite{eda} and \cite{ConnerKentpreprint}.)  Specifically the study of homomorphisms has lead to Eda proving that the fundamental group is a perfect invariant of homotopy type for one-dimensional Peano continua \cite{eda7}.  As well, it is the key tool in the proof that set of points at which a space is not semi-locally simply connected is constructible from the fundamental group for one-dimensional Peano continua \cite{ce} or planar Peano continua \cite{ConnerKentpreprint}. 

In \cite{ConnerKentpreprint}, the author together with Greg Conner show that many of the known results about fundamental groups of one-dimensional spaces extend to planar spaces. Specifically it is proved that any homomorphisms from the fundamental group of a one-dimensional Peano continuum to the fundamental group of a planar Peano continuum is induced by a continuous map after composing with a change of base point isomorphism (Theorem A when $Y$ is a planar Peano continuum).   Here we will prove the following theorem.

\textbf{Theorem \ref{main}} \emph{
Let $\phi:\pi_1(X,x_0) \to\pi_1(Y,y_0)$ be a homomorphism from the fundamental group of a planar Peano continuum $X$ into the fundamental group of a one-dimensional or planar Peano continuum $Y$. Then there exists a continuous function $f:X \to Y$ and a path $\alpha:(I,0)\to (Y,y_0)$, with the property that $f_* =\widehat \alpha \circ\phi$.  In addition, if $\phi$ has uncountable image then $\alpha$ is unique up to homotopy rel endpoints. 
}

In one-dimensional spaces every path class contains a unique minimal representative and every other representative can be homotoped to the unique minimal one be removing backtracking \cite{cc3}.  A loop in a one-dimensional space is called \emph{reduced} if it is the unique minimal representative in its path class.  As well, every one-dimensional Peano continuum deformation retracts to a one-dimensional Peano continuum  in which every point is contained in some homotopically essential reduced loop \cite{cm}.  Two of the difficulties in the planar case is the lack of a canonical deformation retract and the lack of representatives for path classes which are analogous to reduced paths in one-dimensional spaces.


\subsection*{Homotopy dimension}
A space is \emph{homotopically one-dimensional} if it is homotopy equivalent to a one-dimensional space.

Cannon and Conner asked the following question in \cite{cc4}.

\begin{quest}\label{quest1D}
If $X$ is a planar Peano continuum whose fundamental group is isomorphic to the fundamental group of some one-dimensional Peano continuum, is it true that $X$ is homotopic to a one-dimensional Peano continuum?
\end{quest}

Let $\S$ be the Sierpinski curve in $\mathbb R^2$ obtained by the standard Cantor construction performed on the unit square in the plane.  Let $\S_i$ be the planar Peano continuum obtained from $\S$ by filling in $i$ of the removed discs, i.e. $\S_i = \S \cup \bigl(\bigcup\limits_{n=1}^i D_n\bigr)$ where $D_n$ are distinct bounded components of $\mathbb R^2\backslash\S$.  Cannon, Conner and Zastrow showed that $\S_1$  is not homotopy equivalent to any one-dimensional space \cite{ccz}.   Their example, $\S_1$, illustrates that there exist some rigidity in planar sets and at least motivates why Question \ref{quest1D} is interesting.    Karimov, Repovs, Rosicki, and Zastrow give additional examples of planar sets spaces which are not homotopically one-dimensional in \cite{KarimovRepovRosickiZastrow}.

By applying Theorem \ref{main}, we will show that $\S_i$ cannot have the same fundamental group as any one-dimensional Peano continua and that the $\S_i$, $\S_j$ don't have isomorphic fundamental groups for $i\neq j$.  This gives evidence towards an answer to Question \ref{quest1D}.  In a subsequent paper, the author together with Greg Conner uses Theorem \ref{main}  to give a complete answer to Question \ref{quest1D} \cite{ConnerKentpreprint2}.

As an application to Theorem \ref{main} we prove the following result.

\textbf{Corollary \ref{uncountable}}\emph{
There exists an uncountable family of planar Peano continua with pairwise non-isomorphic fundamental groups, all of which are not isomorphic to the fundamental group of any one-dimensional Peano continuum.
}

Our family of examples is constructed by filling infinitely many of the removed squares of $\S$ in a discrete fashion and then studying the limit set of the filled squares.

\section{Planar to one-dimensional or planar}

We will use $\mathbb D$ to denote the unit disc in the Euclidean plane $\mathbb R^2$ and $I$ to denote the interval $[0,1]$.  For a metric space $X$, let $B_r^X(x) = \{ y\in X \ | \ d(x,y)<r\}$ and $S_r^X(x)= \{ y\in X \ |\ d(x,y) = r\}$. For planar sets $X$, $B_r^X(x) = B_r^{\mathbb R^2}(x)\cap X$ and  $S_r^X(x) = S_r^{\mathbb R^2}(x)\cap X$.  

If $X$ is a planar set, we will use $\int{X}$ to denote the interior of $X$ as a subset of the plane, $\cl (X)$ for the closure of $X$ in the plane and  $\partial X$ for $X\backslash \int{X}$.

\begin{thm}\label{basecase}[Eda \cite{eda}, Conner and Kent \cite{ConnerKentpreprint}]
Let $\phi:\pi_1(X,x_0) \to\pi_1(Y,y_0)$ a homomorphism from the fundamental group of a one-dimensional Peano continuum $X$ into the fundamental group of a one-dimensional or planar Peano continuum $Y$. Then there exists a continuous function $f:X \to Y$ and a path $\alpha:(I,0)\to (Y,y_0)$, with the property that $f_* =\widehat \alpha \circ\phi$.  Additionally; if the image of $\phi$ is uncountable, then $\alpha$ is unique up to homotopy rel endpoints.
\end{thm}

\begin{lem}\label{cutoff}
Suppose the $f:\partial \mathbb D \to X$ is a continuous map into a planar or one-dimensional Peano continuum.  If $f$ is null homotopic then there exists a map $h: \mathbb D \to X$ such that $\diam {h(\mathbb D)}\leq \diam{f(\partial\mathbb D)}$.
\end{lem}

The lemma is trivial when $X$ is one-dimensional since every null homotopic loop factors through a dendrite.  Cannon and Conner in \cite{cc3} prove that every null homotopic loop in a planar Peano continuum bounds a disc contained in the convex hull of its image.

\begin{lem}\label{opensquares}
Every bounded open set $U$ of $\mathbb R^2$ is the union of a sequence of dyadic squares with disjoint interiors whose diameters form a null sequence.  In addition, the squares can be chosen such that if $A_i$ is the union of squares with side length at least $\frac{1}{2^i}$, then $U\backslash A_i \subset \mathcal N_{\frac{\sqrt 2}{2^{i-1}}}(\partial U)$.

\end{lem}

This is standard and well known.  We present a proof to introduce notation that we will use later.

\begin{proof}
Let $\chi_i = \bigl\{ (x,y)\mid 0\leq x\leq 1/2^i, \ 0\leq y\leq 1/2^i \bigr\}$ and $Q_i = \bigl\{ (n,m) + \chi_i \mid n,m\in(1/2^i)\mathbb Z\bigr\}$.  So $Q_i$ is the set of tiles in the standard tiling of the plane  by squares with side length $1/2^i$.

Let $D_0$ be the maximal subset of $Q_0$ such that $A_0 \subset U$ where $A_0 = \bigcup\limits_{s\in D_0} s$.  Then $U\backslash A_0 \subset \mathcal N_{\frac{\sqrt 2}{2^{-1}}}(\partial U)$.

We will inductively define $D_i$ and $A_i$ as follows.  Let $D_i$ be the maximal subset of $Q_i$ such that $\bigcup\limits_{s\in D_i} s\subset \cl(U\backslash A_{i-1})$.  Let $A_i = \Bigl(\bigcup\limits_{s\in D_i}s\Bigr)\cup A_{i-1}$.  We immediately have $U\backslash A_i \subset \mathcal N_{\frac{\sqrt 2}{2^{i-1}}}(\partial U)$.  Then $\bigcup\limits_{i=1}^\infty A_i = U$.

\end{proof}

The following lemma is a special case of a lemma proved by Greg Conner and Mark Meilstrup in \cite{cm}.

\begin{lem}\label{aalem1}
Let $h$ be a function from $I\times I$ into a space $Z$.  Let $\{C_i\}$ be a sequence of closed intervals with disjoint interiors which cover $I$.  Suppose that $h$ restricted to $I\times\{0,1\}$ is continuous and, for each $i$, $h$ restricted to $\cl(C_i)\times I$ is continuous.  If $\diam{\bigl\{ h\bigl(\cl(C_i)\times I\bigr)\bigr\}}$ forms a null sequence, then $h$ is continuous.
\end{lem}

\begin{proof}
Consider a sequence $(x_n, y_n) \to (x_0,y_0)$.  For each $n$, choose an $i_n$ such that $x_n\in C_{i_n}$.  If $\{C_{i_n}\}$ is
finite then by restricting $h$ to $\bigl(\cup_n C_{i_n}\bigr)\times Y$ we have $h(x_n,y_n) \to h(x_0,y_0)$ be a finite application of the pasting
lemma.  If $\{C_{i_n}\}$ is infinite, then $\diam{\bigl\{ h\bigl(\cl({C_{i_n}})\times I\bigr)\bigr\}}$ a null
sequence and  $d\bigl(h(x_n,y_n), h(x_n,0)\bigr)$ converges to 0.  Thus $h$ is continuous.
\end{proof}

The following lemma is immediate from the construction of $A_i$ and the diameter condition of the squares composing $A_i$.  Alternatively it is straight forward exercise to show that given a  surjective map $f : I \to X$ how to modify it to construct a surjective map from $I$ to $X^{(1)}$.

\begin{lem}Let $X$ be a planar Peano continuum and $X^{(1)} = \partial X \cup\bigl( \bigcup_i A_i^{(1)}\bigr)$ where $A_i$ is as in Lemma \ref{opensquares} for the bounded open set $\int{X}$.  Then $X^{(1)}$ is a one-dimensional Peano continuum.\end{lem}

\begin{thm}\label{main}
Let $\phi:\pi_1(X,x_0) \to\pi_1(Y,y_0)$ a homomorphism from the fundamental group of a planar Peano continuum $X$ into the fundamental group of a one-dimensional or planar Peano continuum $Y$. Then there exists a continuous function $f:X \to Y$ and a path $\alpha:(I,0)\to (Y,y_0)$, with the property that $f_* =\widehat \alpha \circ\phi$.  In addition, if $\phi$ has uncountable image then $\alpha$ is unique up to homotopy rel endpoints. 
\end{thm}

\begin{proof}

Let  $X^{(1)} = \partial X \cup\bigl( \bigcup_i A_i^{(1)}\bigr)$ where $A_i$ is as in Lemma \ref{opensquares} for the bounded open set $\int{X}$ and $i: X^{(1)}\to X$ be the inclusion map.  Since we are only concerned about the homomorphism up to conjugation, we may assume that $x_0\in X^{(1)}$.

Let $B = \int{X}\backslash X^{(1)}$.  Then $B$ is the disjoint union of open square discs whose diameters form a null sequence.

Fix a loop $\beta: I\to X$ in $X$.  Suppose that $\beta \bigl((t_1,t_2)\bigr)\subset \int{s}$ where $s$ is the closure of some square disc in $ B$.  Then $\beta|_{[t_1,t_2]}$ is homotopic to the line segment $\overline{\beta(t_1)\beta(t_2)}$ by a homotopy contained entirely in $\overline s$.  Let $\beta'$ be the path obtained from $\beta$ by mapping each maximal open interval in $\beta^{-1}\bigl( B\bigr)$ to the line segment connecting its endpoints. Lemma \ref{aalem1} implies that $\beta$ is homotopic to $\beta'$ which implies that $i_*$ is surjective.

By Theorem \ref{basecase}; $\phi\circ i_*: \pi_1(X^{(1)},x_0) \to \pi_1(Y,y_0)$ is conjugate to being induced by a continuous map, i.e.  $\phi\circ i_* = \widehat {\overline \alpha}\circ f_*$  where $f:X^{(1)}\to Y$ is a continuous map and $\alpha: I \to Y$ is a continuous path.

Let $s$ be the closure of a component of $B$, i.e. a square from our construction of $X^{(1)}$.  Then $f|_{\partial s}$ is a null homotopic loop in $Y$.  Thus we can extend $f$ to all of $s$ such that $\diam{\bigl(f(s)\bigr)}\leq \diam {\bigl(f(\partial s)\bigr)}$.  Doing this for all the components of $B$ defines an extension $\overline f$ of $f$ to all of $X$.  The diameter condition guarantees the continuity of $\overline f$.

Let $\beta $ be a loop in $X$.  Then there exists a loop $\beta'$ in $X^{(1)}$ homotopic (in $X$) to $\beta$.  Then
\begin{align*}\phi([\beta]) &= \phi\circ i_*([\beta']) = \widehat {\overline \alpha}\circ f_*([\beta']) \\
&= \widehat {\overline \alpha}\bigl([ f\circ \beta']\bigr)=\widehat {\overline \alpha}\bigl([ \overline f\circ \beta']\bigr)\\
&= \widehat {\overline \alpha}\bigl([ \overline f\circ \beta]\bigr)= \widehat {\overline \alpha}\circ \overline f_*\bigl([\beta]\bigr)
\end{align*}

as desired.

\end{proof}

\section*{Applications}
\setcounter{section}{2}

Let $C = \Bigl([0,1]\times[0,1]\Bigl)\backslash \Bigl((1/3,2/3) \times (1/3,2/3) \Bigl)$.  Then $\S$, the standard Sierpinski curve in $\mathbb R^2$, is constructed by iterating the process of replacing $[0,1]\times [0,1]$ with a copy of $C$ and subdividing $C$ into $8$ squares with side length $1/3$ of the original square. 

Then $\mathbb R^2\backslash \S$ is the union of countable many open squares with disjoint closures and a single unbounded component.  Let $\bigl\{D_n\bigr\}$ be an enumeration of the open squares contained in the complement of $\S$.

For $A\subset \mathbb N$, let $\S_A = \S \cup \Bigl(\bigcup\limits_{n\in A} D_n\Bigr)$; i.e. $\S_A$ is the space obtained from $\S$ by filling in the squares corresponding to $A$.  For $i\in \mathbb N$, let $\S_i = \S \cup \Bigl(\bigcup\limits_{n=1}^i D_n\Bigr)$.  
  
The following is well-know and a proof can be found in \cite{ConnerKentpreprint}.  
  
\begin{lem}\label{fix}
Suppose that $h:X\to X$ is a continuous map of a planar Peano continuum such that every loop is freely homotopic to its image under $h$. Then $h$ fixes the set points at which $X$ is not semi-locally simply connected.
\end{lem}

\begin{proof}[Sketch of proof.]
Suppose that  $X$ is not semi-locally simply connected at $x$ and $h(x)\neq x$.  Then we would be able to find an $\epsilon>0$  such that the balls $B_\epsilon(x)$ and $B_\epsilon\bigl(h(x)\bigr)$ are disjoint and  $S_\epsilon^X(x) \subsetneq S_\epsilon^{\mathbb R^2}(x)$.  This implies that $S_\epsilon^X(x)$ is the disjoint union of closed intervals.  

Since any loop is freely homotopic to its image under $h$, any sufficiently small loop in $B_\epsilon(x)$ can be homotoped into $B_\epsilon\bigl(h(x)\bigr)$.  However, any map of an annulus which takes one boundary component into $B_\epsilon(x)$ and the other into $B_\epsilon\bigl(h(x)\bigr)$ can by cut along $S_\epsilon^X(x)$ to construct a null homotopy of the boundary loops. Thus any sufficiently small loop in $B_\epsilon(x)$ must be null homotopic.  However this contradicts the assumption that $X$ is not semilocally simply connected at $x$. 

\end{proof}

Cannon, Conner, and Zastrow showed that $\S_1$ is not homotopy equivalent to a one-dimensional Peano continuum.  We can now use Theorem \ref{main}  to show even more that the fundamental group of $\S_1$ is not \emph{one-dimensional} in the following sense.

\begin{thm}\label{onefilled} For any $s_0\in \S$, $\pi_1(\S_i, s_0)$ is not isomorphic to the fundamental group of any one-dimensional Peano continuum.
\end{thm}

\begin{proof}
Suppose that there exists $X$ a one-dimensional Peano continuum such that $\pi_1(X,x_0)$ is isomorphic to $\pi_1(\S_i,s_0)$.  By Theorem \ref{main}, there exists a continuous map $f: \S_i \to X$ which induces an isomorphism $f_*$ of fundamental groups.

By applying Theorem \ref{basecase} to the homomorphism $f_*^{-1}$, we can find a map $g: X \to \S_1$ such that $g\circ f\circ \beta$ is freely homotopic to $\beta$ for ever loop $\beta$ based at $x_0$.  (Note that $\beta$ might not be homotopic to $g\circ f\circ \beta$ relative to endpoints.)  

 Since $\textbf S_i$ is obtained by only adding finitely many discs, every neighborhood of every point in $\textbf S$ contains a loop which is essential in $\textbf S_i$. Thus $g\circ f$ must fix $\S$ by Lemma \ref{fix}.

Let $D_k$ be a square which was filled in the construction of $\S_i$. Since $f$ maps $\partial D_k$ to a null homotopic loop in a one-dimensional space, the map $f$ must identify two distinct points $x,y$ on the boundary of $D_k$.  However this is a contradiction since $\partial D_k\subset \S$.

\end{proof}

\begin{cor}\label{finitedifferent}
 For any $s_0\in \S$ and any  pair of distinct natural numbers $i,j$; $\pi_1(\S_i, s_0)$ is not isomorphic to $\pi_1(\S_j, s_0)$.
\end{cor}

\begin{proof}
We will assume that $i>j$ and proceed by way of contradiction.  As in the proof of Theorem \ref{onefilled}, we may assume that there are maps $f: \S_i \to \S_j $ and $g: \S_j \to \S_i $ such that $g\circ f\circ \beta$ is freely homotopic to $\beta$ for any loop $\beta$ based at $s_0$.  As before $g\circ f$ must fix $\S$.  

Let $D_k$ be a square which was filled in the construction of $\S_i$.  Notice that $\partial D_k$ must map to a simple  closed curve which is null homotopic in $\S_j$ ($f|_{\textbf S}$ must be injective).  Hence it must map to the boundary of a square which was filled in the construction of $\S_j$.  (A simple closed curve $\alpha$ in the plane is null homotopic if an only if the bounded component of $\mathbb R^2\backslash \im(\alpha)$ is simply connected.)   Since $j>i$, $f$ must map two boundary circles to the same boundary circle which contradicts that fact that $f$ restricted to $\S$ must be injective.

\end{proof}

We will now show how to extend Corollary \ref{finitedifferent} to  certain nice fillings of $\S$.

\begin{defn}\label{Bad/Odd sets}
Let $A\subset \mathbb N$.  We will use $B(\S_A)$ to denote the set of points at which $\S_A$ is not semi-locally simply connected.  Let $K\bigl(\S_A\bigr)$ be the set of accumulation points of $\bigl\{D_n\mid n\in A\bigr\}$, i.e. the set of points such that every neighborhood contains an element of $\bigcup\limits_{n\in A} D_n$.  

We will say that $\S_A$ is a \emph{discrete} filling of $\S$ if $\cl(D_n)\cap K(\S_A) = \emptyset $ for all $n\in A$. 

\end{defn}

The follow lemma is proved similarly to Theorem \ref{onefilled}

\begin{lem}\label{bad set fillings}
If $\S_A$ is a discrete filling then $\pi_1(\S_A, s)$ is not isomorphic to the fundamental group of a one-dimensional Peano continuum.
\end{lem}

\begin{lem}
Let $Y$ be a closed subset of $\S$.  Then there exists a subset $A\subset \mathbb N$ such that $\S_A$ is a discrete filling of $\S$ and $K(\S_A) = Y$.
\end{lem}

\begin{proof}    Let $\mathcal U_n$ be a minimal finite cover of $Y$ by balls in $\mathbb R^2$ of radius $1/n$ (minimal in the sense that no subset of $\mathcal U_n$ covers $Y$).  
Choose $\delta_n$ such that the ball in $\mathbb R^2$ of radius $\delta_n$ about a point in $Y$ is contained in a single element of $\mathcal U_n$. 

Let $A_i$ be a set of indices such that each element of $\mathcal U_i$ contains exactly one square from $\{D_n\}$ with index in $ A_i$ and that square does not intersect $\mathcal N_{\delta_i/2}( Y)$.   Let $A = \cup_i A_i$. Then $\S_A$ is a discrete filling of $\S$ and $K(\S_A) = Y$.
\end{proof}

\begin{lem}\label{bad set fillings}
If $\S_A$ is a discrete filling then $\partial D_n \subset B(\S_A)$ for all $n\in A$  and $B(\S_A) = \S$.
\end{lem}

\begin{proof}
It is clear that $B(\S_A)\subset \S$.  By construction, $\cl(D_n)\cap \cl(D_m) = \emptyset $ for all $n\neq m$.  
For $n\in A$, let $\epsilon_n$ be the distance from $\cl(D_n)$ to $K(\S_A)\cup\left(\bigcup\limits_{ i\in A\backslash\{n\}} D_i\right)$.  Since $\cl(D_n)\cap K(\S_A) = \emptyset $,  $\epsilon_n$ is strictly positive.  This implies that $\mathcal N_{\epsilon_n}(D_n)$ is not simply connected.  Even more $B\bigl(\mathcal N_{\epsilon_n}(D_n)\bigr) = \mathcal N_{\epsilon_n}(D_n)\backslash D_n$.  Thus the only points of $\S$ which might possible have simply connected neighborhoods in $\S_A$ are those in $K(\S_A)$.

Suppose that $x\in \S\cap K(\S_A)$ and let $U$ be a neighborhood of $x$.  We must show that $U$ is not simply connected.  Since $x\in K(\S_A)$, we can find $n\in A$ such that $\cl(D_n)\subset U$.  Therefore $\mathcal N_{\epsilon}(D_n)\subset U$ for some choice of $n\in A$ and $0<\epsilon\leq \epsilon_n$ which implies that $U$ is not simply connected. 

\end{proof}

\begin{prop}\label{nonisomorphic} Suppose that $\S_A$, $\S_B$ are discrete fillings of $\S$.  If $\pi_1(\S_A, s_0)$ is isomorphic to $\pi_1(\S_B, s_0)$, then $K(\S_A)$ is homeomorphic to $K(\S_B)$. 
\end{prop}

\begin{proof}
Suppose that $\S_A$, $\S_B$ are discrete fillings of $\S$ and $\pi_1(\S_A, s_0)$ is isomorphic to $\pi_1(\S_B, s_0)$.  Using Theorem \ref{main} and Lemma \ref{fix}, we can find maps $f: \S_A \to \S_B$ and $g:\S_B\to S_A$ such that $g\circ f$ is the identity on $\S$.  (For discrete fillings $B(\S_A) = \S$ by Lemma \ref{bad set fillings}.)  For $n\in A$ $\partial D_n$ is a null homotopic simple closed curve in $\S\subset \S_A$ which implies that  $f(\partial D_n)$ is a null homotopic simple closed curve in $\S_B$.  (A simple closed curve $\alpha$ in the plane is null homotopic if an only if the bounded component of $\mathbb R^2\backslash \im(\alpha)$ is simply connected.)  Thus $f(\partial D_n) = D_m$ for some $m\in B$.

Thus $f\bigl(K(\S_A)\bigr) \subset K(\S_B)$ and similarly $g\bigr(K(\S_B)\bigl)\subset K(\S_A)$.  Since $\S_A,\S_B\subset \S$  and $g\circ f$ is the identity on $\S$ implies that $K(\S_A)$ is homeomorphic to $K(\S_B)$.

\end{proof}

\begin{cor}\label{uncountable}
There exists an uncountable family of planar Peano continua with pairwise non-isomorphic fundamental groups, all of which are not isomorphic to the fundamental group of any one-dimensional Peano continuum.
\end{cor}

\begin{proof}
Let $\{U_1,U_2,\cdots\}$ be a countable set of disjoint open subsets of $(0,1)\times (0,1)$.  In each $U_i$ we can find a subset $X_i$ such that $X_i\subset\S$ and $X_i$ is homeomorphic to the wedge of $i$-closed intervals.  For $\bar A\subset \mathbb N$, let $X_{\bar A} = \bigcup\limits_{i\in \bar A} X_i$.  It is a trivial exercise to show that that $ X_{\bar A}$ is homeomorphic to $X_{\bar B}$ if and only if $\bar A = \bar B$.

For $\bar A\subset \mathbb N$, choose $A\subset \mathbb N$ such that $K(\S_A) = X_{\bar A}$.  The corollary then follows from Proposition \ref{nonisomorphic}.
\end{proof}

\bibliographystyle{plain}
\bibliography{/home/ckent/NYU_Papers/bib}

\end{document}